\documentclass[11pt]{article}
\usepackage{amssymb,latexsym,amsmath,amsbsy,amsthm,amsxtra,amsgen,graphicx,dsfont,color}
\newfont{\msbm}{msbm10 at 11pt}

\newcommand\QED{\pushQED{\qed} 
\qedhere
\popQED}

\newcommand\Quote[1]{``#1"}

\newcommand {\PP} {\mathbb{P}}

\def\eps{\varepsilon}

\newtheorem{Theo}{Theorem}

\newtheorem{Rmk}{Remark}

\newtheorem{Lem}{Lemma}

\begin{document}

\title{A Spatial Mutation Model with Increasing Mutation Rates}
\author{Brian Chao and Jason Schweinsberg \\ University of California San Diego}
\maketitle
\begin{abstract}
We consider a spatial model of cancer in which cells are points on the $d$-dimensional torus $\mathcal{T}=[0,L]^d$, and each cell with $k-1$ mutations acquires a $k$th mutation at rate $\mu_k$. We will assume that the mutation rates $\mu_k$ are increasing, and we find the asymptotic waiting time for the first cell to acquire $k$ mutations as the torus volume tends to infinity. This paper generalizes results on waiting for $k\geq 3$ mutations by Foo, Leder, and Schweinsberg \cite{fls}, who considered the case in which all of the mutation rates $\mu_k$ were the same. In addition, we find the limiting distribution of the spatial distances between mutations for certain values of the mutation rates.
\end{abstract}

{\small {{\it AMS 2020 subject classifications}:  Primary 60J99;
Secondary 60G55, 92D15, 92D25}

}

\section{Introduction} 

Cancer is often caused by genetic mutations which disrupt regular cell division and apoptosis, in which case cancerous cells divide too rapidly and healthy cells reproduce too slowly. This can happen, for example, as soon as several distinct mutations occur and dramatically disrupt cell function. Thus, it is sometimes reasonable to model cancer as occurring after $k$ distinct mutations appear in sequence within a large body.

Mathematical models in which cancer occurs once some cell acquires $k$ mutations date back to the famous 1954 paper by Armitage and Doll \cite{armdoll}.  Armitage and Doll proposed a multi-stage model of carcinogenesis in which, once a cell has acquired $k-1$ mutations, it acquires a $k$th mutation at rate $\mu_k$.  In this model, the probability of acquiring the $k$th mutation during a small time interval $(t,t+dt)$ is
$$\frac{\mu_1\mu_2\cdots\mu_kt^{k-1}}{(k-1)!}dt.$$
That is, the incidence rate of the $k$th mutation (at which point the individual becomes cancerous) is proportional to $\mu_1\mu_2\cdots\mu_k t^{k-1}$. This means that cancer risk is proportional to both the mutation rates and the $(k-1)$th power of age.  More sophisticated models, taking into account the possibilities of cell division and cell death, were later analyzed in \cite{dm11, dm10, dss, imkn, imn, ksn, ml90, ml92, sch08}.

To model some types of cancer, it is important also to include spatial structure in the model.  In 1972, Williams and Bjerknes \cite{wb72} introduced a spatial model of skin cancer, now known as the biased voter model.  At each site on a lattice, there is an associated binary state indicating whether the site is cancerous or healthy.  Each cell divides at a certain rate, and when cell division occurs, the daughter cell replaces one of the neighboring cells chosen at random.  The model is biased in that a cancerous cell spreads $\kappa>1$ times as fast as a healthy cell.  Williams and Bjerknes \cite{wb72} presented computer simulations for this model, and the model was later analyzed mathematically by Bramson and Griffeath \cite{bg1, bg2}.

More recently, Durrett, Foo, and Leder \cite{dfl16}, building on earlier work in \cite{dm15, k06}, studied a spatial Moran model which is a generalization of the biased voter model.  Cells are modelled as points of the discrete torus $(\mathbb{Z}\text{ mod }L)^d$, and each cell is of type $i\in\mathbb{N}\cup\{0\}$.  A cell of type $i-1$ mutates to type $i$ at rate $\mu_i$.  Type $i$ cells have fitness level $(1+s)^i$, where $s>0$ measures the selective advantage of one cell over its predecessors.  Each cell divides at a rate proportional to its fitness, and then, as in the biased voter model, the daughter cell replaces a randomly chosen neighboring cell.  The authors considered the question of how long it takes for some type 2 cell to appear.  To simplify the analysis, they introduced a continuous model where cells live inside the torus $[0,L]^d$. This continuous stochastic model approximates the biased voter model because of the Bramson-Griffeath shape theorem \cite{bg1,bg2}, which implies that the cluster of cells in $\mathbb{Z}^d$ with a particular mutation grows to the shape of a convex subset of $\mathbb{R}^d$.  In Section 4 of \cite{dfl16}, the authors used the continuous model to compute the distribution of the time that the first type 2 cell appears, under certain assumptions on the mutation rates. 

We describe here in more detail this continuous approximation to the biased voter model.  The spread of cancer is modeled on the $d$-dimensional torus $\mathcal{T}:=[0,L]^d$, where the points $0$ and $L$ are identified.  Note that this is the continuous analog of the space $(\mathbb{Z}\text{ mod }L)^d$ considered in \cite{dfl16}.  We write $N:=L^d$ to denote the volume of $\mathcal{T}$.  Each point in $\mathcal{T}$ is assigned a type, indicating the number of mutations the cell has acquired.  At the initial time $t=0$, all points in $\mathcal{T}$ are type~$0$, meaning they have no mutations.  A so-called type $1$ mutation then occurs at rate $\mu_1$ per unit volume. Once each type $1$ mutation appears, it spreads out in a ball at rate $\alpha$ per unit time.  This means that $t$ time units after a mutation appears, all points within a distance $\alpha t$ of the site where the mutation occurred will have acquired the mutation.  Type $1$ points then acquire a type $2$ mutation at rate $\mu_2$ per unit volume, and this process continues indefinitely.  In general, type $k$ mutations overtake type $k-1$ mutations at rate $\mu_k$ per unit volume, and each type $k$ mutation then spreads outward in a ball at rate $\alpha$ per unit time.  A full mathematical construction of this process, starting from Poisson point processes which govern the mutations, is given at the beginning of section \ref{waitkproofs}.  

Let $\sigma_k$ denote the first time that some cell becomes type $k$.  Foo, Leder, and Schweinsberg \cite{fls} obtained the asymptotic distribution of $\sigma_2$ under a wide range of values for the parameters $\alpha$, $\mu_1$, and $\mu_2$, extending the results in \cite{dfl16}.  They also found the asymptotic distribution of $\sigma_k$ for $k \geq 3$ assuming equal mutation rates $\mu_i=\mu$ for all $i$.  In this paper, we will further generalize the results in \cite{fls} for $k \geq 3$ by considering the case where the mutation rates are increasing.  We will see that several qualitatively different types of behavior are possible, depending on how fast the mutation rates increase.

We mention two biological justifications for assuming increasing mutation rates. Loeb and Loeb \cite{ll00} suggested a general phenomenon in carcinogenesis where there is favorable selection for certain mutations that promote tumor growth in genes responsible for repairing DNA damage. The increasing genetic instability disrupting DNA repair, in the context of the present paper, would correspond to increasing mutation rates. Also, our model would be of interest in the situation described by Prindle, Fox, and Loeb \cite{pfl10}, who hypothesize that cancer cells express a mutator phenotype, which causes cells to mutate at a much higher rate.  They propose targeting the mutator phenotype as part of cancer therapy, possibly with the goal of further increasing the mutation rate to the point where the mutations incapacitate or kill malignant cells.


As in \cite{fls}, we will assume that the rate of mutation spread $\alpha$ is constant across mutation types, so that successive mutations have equal selective advantage. One possible generalization of our model would be to allow each type $i$ mutation to have a different rate of spread $\alpha_i$.  However, this more general model is nontrivial even to formulate unless $(\alpha_i)_{i=1}^{\infty}$ is decreasing because if $\alpha_{i+1}>\alpha_{i}$, then regions of type $i+1$ could completely swallow the surrounding type $i$ region.  Consequently, it would be necessary to model what happens not only when mutations of types $i+1$ and $i$ compete, but also how mutations of types $i+1$ and $j\in \{1,...,i-1\}$ compete.  We do not pursue this generalization here.

After computing the limiting distribution of $\sigma_k$, we also find the limiting distribution of the distances between the first mutation of type $i$ and the first mutation of type $j$, where $i < j$.
The distribution of distances between mutations is relevant in studying a phenomenon known as the \Quote{cancer field effect}, which refers to the increased risk for certain regions to acquire primary tumors. These regions are called premalignant fields, and they have a high risk of becoming malignant despite appearing to be normal \cite{flr}. The size of the premalignant field is clinically relevant when a patient is diagnosed with cancer, because it will determine the area of tissue to be surgically removed, in order to avoid cancer recurrence. Surgical removal of premalignant fields, put in the context of this paper, is akin to removing the region with at least $i$ mutations once the first type $j>i$ mutation appears.  Foo, Leder, and Ryser \cite{flr} considered the case in which $i = 1$ and $j = 2$, and they characterized the sizes of premalignant fields conditioned on $\{\sigma_2=t\}$, in $d\in \{1,2,3\}$ spatial dimensions.  These ideas were applied to head and neck cancer in \cite{rlrlf16}.

We note that the model that we are studying in this paper independently appeared in the statistical physics literature, where it is known as the polynuclear growth model.  It has been studied most extensively in $d = 1$ when all of the $\mu_k$ are the same \cite{ps00a, ps00b, bfs08}, but the model was also formulated in higher dimensions in \cite{ps02}.  Most of this work in the statistical physics literature focuses on the long-run growth properties of the surface, and detailed information about the fluctuations has been established when $d = 1$.  This is quite different from our goal of understanding the time to acquire a fixed number of mutations.

In Section 2, we introduce some basic notation and state our main results as well as some heuristics explaining why these results are true.  In Section \ref{waitkproofs}, we prove the limit theorems regarding the time to wait for $k$ mutations, and in Section \ref{distanceproof}, we prove the limit theorems for the distances between mutations.

\section{Main results and heuristics}\label{waitingresults}

We first introduce some notation that we will need before stating the results.
Given two sequences of nonnegative real numbers $(a_N)_{N=1}^{\infty}$ and $(b_N)_{N=1}^{\infty}$, we write:
\begin{enumerate}
    \item $a_N\sim b_N$ if $\lim_{N\to\infty}a_N/b_N=1$;
    \item $a_N\ll b_N$ if $\displaystyle \lim_{N\to \infty}a_N/b_N=0$ and $a_N\gg b_N$ if $\lim_{N\to\infty}a_N/b_N=\infty$;
    \item $a_N\asymp b_N$ if $\displaystyle 0<\liminf_{N\to \infty}a_N/b_N\leq \limsup_{N\to \infty}a_N/b_N<\infty$;
    \item $a_N\lesssim b_N$ if $\displaystyle \limsup_{N\to\infty} a_N/b_N<\infty$.
\end{enumerate}
 We also define the following notation:
\begin{enumerate}
    \item[a.] If $X_N$ converges to $X$ in distribution, we write $X_N\Rightarrow X$.
    \item[b.] If $X_N$ converges to $X$ in probability, we write $X_N\to_p X$.
    \item[c.] $\gamma_d$ denotes the volume of the unit ball in $\mathbb{R}^d$.
    \item[d.] For each $k\geq 1$ and $j \geq 1$, we define
    \begin{align}
        \label{betakay}
        \beta_k:=\Big(N\alpha^{(k-1)d}\prod_{i=1}^{k}\mu_i\Big)^{-1/((k-1)d+k)} \text{ and }\kappa_j:=(\mu_j\alpha^d)^{-1/(d+1)}.
    \end{align} 
    We will explain how $\beta_k$ and $\kappa_j$ arise in Sections 2.3 and 2.5, respectively.
    \item[e.] $\sigma_k$ denotes the first time a mutation of type $k$ appears, and $\sigma_k^{(2)}$ denotes the second time a mutation of type $k$ appears. More rigorous definitions of $\sigma_k$ and $\sigma_k^{(2)}$ are given in Sections~\ref{waitkproofs} and \ref{distanceproof}, respectively. 
\end{enumerate}
All limits in this paper will be taken as $N \rightarrow \infty$.
The mutation rates $(\mu_i)_{i=1}^{\infty}$ and the rate of mutation spread $\alpha$ will depend on $N$, even though this dependence is not recorded in the notation. Throughout this paper we will assume that the mutation rates $(\mu_i)_{i=1}^{\infty}$ are asymptotically increasing, i.e.
\begin{align}
    \label{increasing}
    \mu_1\lesssim \mu_2\lesssim \mu_3\lesssim \cdots
\end{align}

\subsection{Theorem 1: low mutation rates} Assume
$$\displaystyle \mu_1\ll \frac{\alpha}{N^{(d+1)/d}}\text{ and }  \frac{\mu_i}{\mu_1}\to c_i\in (0,\infty] \text{ for all }i\in \{1,...,k\}.$$
The first time a mutation of type $1$ appears is exponentially distributed with rate $N\mu_1$. The maximal distance between any two points on the torus $\mathcal{T}=[0,L]^d$ is $\sqrt{d}L/2$. Also note that $L=N^{1/d}$ where $N$ is the volume of $\mathcal{T}$. Consequently, once the first type 1 mutation appears, it will spread to the entire torus in time $\sqrt{d}L/(2\alpha)=\sqrt{d}N^{1/d}/(2\alpha)$. Hence, as noted in \cite{fls}, the time required for a type $1$ mutation to fixate once it has first appeared is much shorter than $\sigma_1$ precisely when $N^{1/d}/\alpha\ll 1/(N\mu_1)$, which is equivalent to $\mu_1\ll \alpha/N^{(d+1)/d}$.

Now because of the second assumption $\mu_i/\mu_1\to c_i \in (0,\infty]$, mutations of types $i\in \{2,...,k\}$ appear at least as fast as the first mutation. If $c_i<\infty$, then the waiting times  $\sigma_1$ and $\sigma_i-\sigma_{i-1}$ are on the same order of magnitude. Because we have $\sigma_1\sim \text{Exponential}(N\mu_1c_1)$, it follows $\sigma_i-\sigma_{i-1}$ is also exponentially distributed and that $\sigma_i-\sigma_{i-1}\sim\text{Exponential}(N\mu_1c_i)$. Otherwise, if $c_i=\infty$, then the first type $i$ mutation appears so quickly that its waiting time $\sigma_i-\sigma_{i-1}$ is negligible as $N\to\infty$. Putting everything together gives us the following theorem.  This result is a very slight generalization of Theorem 1 in \cite{fls}, and is proved by the same method.

\begin{Theo}
Suppose (\ref{increasing}) holds and $\mu_1 \ll \alpha/N^{(d+1)/d}$. Suppose that for all $i\in \{1,...,k\}$, we have
$$\frac{\mu_i}{\mu_1}\to c_i\in (0,\infty].$$
Let $W_1,...,W_k$ be independent random variables with $W_i\sim \text{Exponential}(c_i)$ if $c_i<\infty$ and $W_i=0$ if $c_i=\infty$. Then
$$N\mu_1\sigma_k\Rightarrow W_1+\cdots+W_k.$$
\end{Theo}

Figure 1 below illustrates that once a type $i$ mutation appears, it quickly fills up the whole torus, and then a type $i+1$ mutation occurs. 
\begin{center}
\includegraphics[scale=0.33]{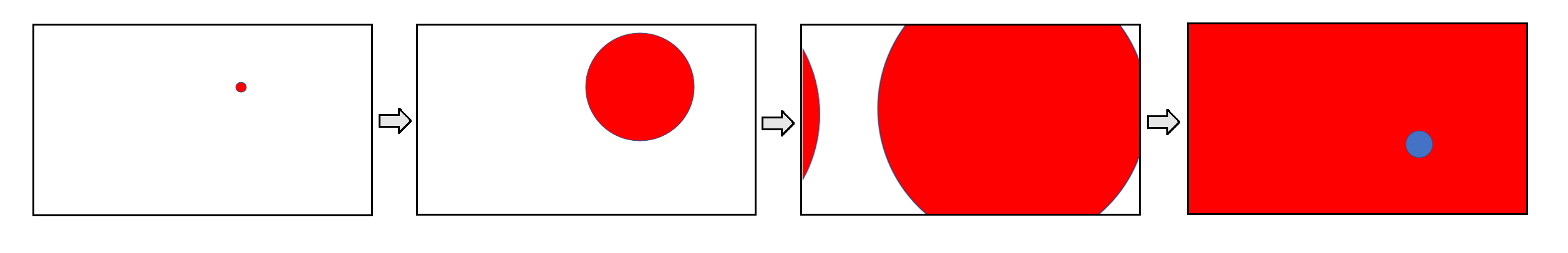}
\\ Figure 1. The transition from type $i$ mutations (in red) to type $i+1$ mutations (in blue). 
\end{center}

\subsection{Theorem 2: type $j\geq 2$ mutations occur rapidly after $\sigma_1$} 

Assume
\begin{equation}\label{asmthm2}
\displaystyle \mu_1\gg \frac{\alpha}{N^{(d+1)/d}} \text{ and }\displaystyle \mu_2\gg \frac{(N\mu_1)^{d+1}}{\alpha^d}.
\end{equation} 
In contrast to Theorem 1, the assumption $\mu_1\gg \alpha/N^{(d+1)/d}$ means that the time it takes for type~$1$ mutations to spread to the entire torus is much longer than $\sigma_1$.  As a result, there will be many growing balls of type 1 mutations before any of these balls can fill the entire torus.  However, if mutations of types $2, 3, \dots, k$ appear quickly after the first type 1 mutation appears, then the time to wait for the first type $k$ mutation will be close to the time to wait for the first type 1 mutation.  We consider here the conditions under which this will be the case.

First consider the ball of type 1 cells resulting from the initial type 1 mutation at time $\sigma_1$.  Assuming $t$ is small enough that, by time $\sigma_1 + t$, the ball has not started overlapping itself by wrapping around the torus, the ball will have volume $\gamma_d (\alpha t)^d$ at time $t$.  Then the probability that the first type $2$ mutation appears in that ball before time $t$ is
\begin{align}
    \label{singleball}
    1-\exp\Big(-\int_0^{t}\mu_2\gamma_d(\alpha r)^d dr \Big)=1-\exp\Big(-\frac{\gamma_d}{d+1}\mu_2\alpha^d t^{d+1}\Big).
\end{align}
It follows that the first time a type 2 mutation occurs in this ball is on the order of $(\mu_2\alpha^d)^{-1/(d+1)}$. 
Hence, whenever $(\mu_2\alpha^d)^{-1/(d+1)}\ll 1/(N\mu_1)$, which is equivalent to the second assumption in (\ref{asmthm2}), it follows that $\sigma_2-\sigma_1$ is much quicker than $\sigma_1$. From this heuristic, we see that $N\mu_1(\sigma_2-\sigma_1)\to_p 0$. 
Repeating this reasoning with types $j-1$ and $j$ in place of types $1$ and $2$, we see that $\sigma_j-\sigma_{j-1}$ is much quicker than $\sigma_1$ when $(\mu_j\alpha^d)^{-1/(d+1)}\ll 1/(N\mu_1)$, or equivalently $\mu_j\gg (N\mu_1)^{d+1}/\alpha^d$. However, this follows from the second assumption in (\ref{asmthm2}) because of (\ref{increasing}). Hence, we also have $N\mu_1(\sigma_j-\sigma_{j-1})\to_p 0$.  
Putting everything together, when $N$ is large,
$$N\mu_1\sigma_k= N\mu_1\sigma_1+N\mu_1(\sigma_2-\sigma_1)+\cdots+N\mu_1(\sigma_k-\sigma_{k-1})\approx N\mu_1\sigma_1.$$
This gives us the following theorem.  We note that the $k=2$ case was proven by Durrett, Foo, and Leder in Theorem 3 of \cite{dfl16}.  They used essentially the same reasoning that is described above.

\begin{Theo}
\label{theorem2}
Suppose (\ref{increasing}) holds, and suppose $\mu_1\gg \alpha/N^{(d+1)/d}$ and $\mu_2\gg (N\mu_1)^{d+1}/\alpha^d$.  For all $k \geq 2$,
$$N\mu_1\sigma_k\Rightarrow W,$$
where $W\sim \text{Exponential}(1)$.
\end{Theo}

A pictorial representation is given in Figure 2, where the nested circles correspond to mutations of types $1,...,k$ for $k=4$.
\begin{center}
    \includegraphics[scale=0.4]{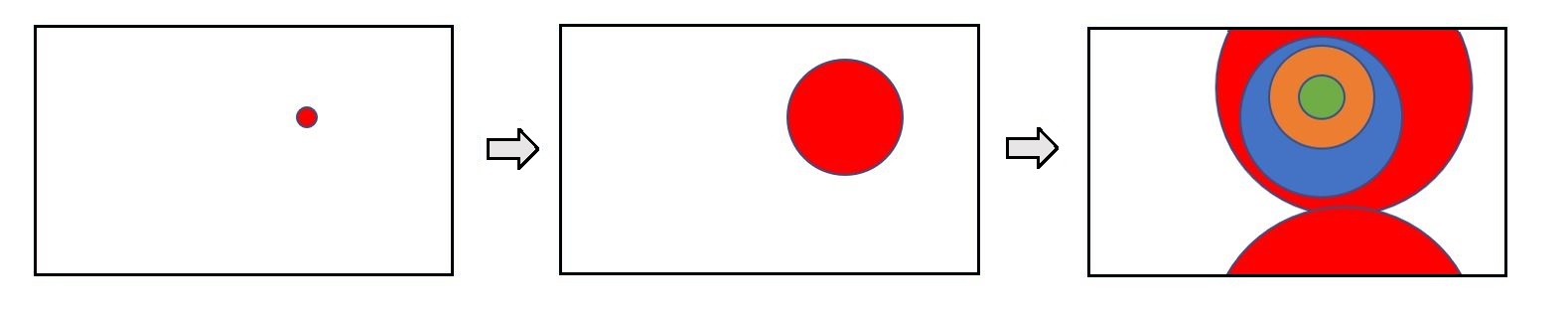}
    \\ Figure 2. Once the first type $1$ mutation (red) appears, the type 2 (blue), type 3 (orange), and type 4 (green) mutations all happen quickly.
\end{center}

\subsection{Theorem 3: type $j\in \{1,...,k-1\}$ mutations appear many times}
Assume
\begin{equation}\label{asmthm3}
\mu_1\gg \frac{\alpha}{N^{(d+1)/d}} \text{ and } \mu_k\ll \frac{1}{\alpha^d\beta^{d+1}_{k-1}}.
\end{equation}
As in Theorem 2, the first assumption ensures that $\sigma_1$ is shorter than the time it takes for type 1 mutations to fixate once they appear. The second assumption ensures that all mutations of types up to $k$ do not appear too quickly, so that we are not in the setting of Theorem 2.  In particular, note that when $k=2$, we have $\beta_{k-1}=(N\mu_1)^{-1}$, and the second assumption reduces to $\mu_2\ll (N\mu_1)^{d+1}/\alpha^d$.  When (\ref{asmthm3}) holds, for $j \in \{2, \dots, k\}$, there will be many small balls of type $j-1$ before any type $j$ mutation appears.  In this case, we will be able to use a ``law of large numbers" established in \cite{fls} to approximate the total volume of type $j-1$ regions with its expectation.


To explain what happens in this case, we review a derivation from \cite{fls}. We want to define an approximation $v_j(t)$ to the total volume of regions with at least $j$ mutations at time $t$. We set $v_0(t)\equiv N$. Next, let $t>0$. For times $r\in [0,t]$, type $j$ mutations occur at rate $\mu_jv_{j-1}(r)$, and these type $j$ mutations each grow into a ball of size $\gamma_d(\alpha(t-r))^d$ by time $t$. Therefore, we define
\begin{equation}\label{informal}
v_j(t) =\int_{0}^{t}\mu_jv_{j-1}(r)\gamma_d(\alpha(t-r))^d dr.
\end{equation}
Note that (\ref{informal}) gives a good approximation to the volume of the type $j$ region because we have many mostly non-overlapping balls of type $j$.
In \cite{fls} it is shown using induction that
$$v_j(t)=\frac{\gamma_d^j(d!)^j}{(j(d+1))!}\Big(\prod_{i=1}^{j}\mu_i\Big)N\alpha^{jd}t^{j(d+1)},$$
which gives us the approximation
$$\mathbb{P}(\sigma_k>t)\approx \exp\Big(-\int_{0}^{t}\mu_kv_{k-1}(r)dr\Big)=\exp\Big(-\frac{\gamma_d^{k-1}(d!)^{k-1}}{((k-1)d+k)!}\Big(\prod_{i=1}^{k}\mu_i\Big)N\alpha^{(k-1)d}t^{(k-1)d+k}\Big).$$
It will follow that if we define $\beta_k$ as in (\ref{betakay}),
then we have the following result.

\begin{Theo}
\label{theorem3}
Suppose (\ref{increasing}) holds.  Let $k\geq 2$, and suppose $\mu_1\gg \alpha/N^{(d+1)/d}$ and $\mu_k\ll 1/(\alpha^d\beta_{k-1}^{d+1})$.
Then for $t>0$, $$\mathbb{P}(\sigma_k>\beta_kt)\to \exp\Big(-\frac{\gamma_d^{k-1}(d!)^{k-1}}{((k-1)d+k)!}t^{(k-1)d+k}\Big).$$
\end{Theo}
When we have equal mutation rates (i.e. $\mu_i=\mu$ for all $i$), the result above is covered by part~3 of Theorem 10 in \cite{fls}.  The form of the result and the strategy of the proof are exactly the same in the more general case when the mutation rates can differ.  Theorem 3 is illustrated in Figure 3 for $k=3$.
\begin{center}
    \includegraphics[scale=0.33]{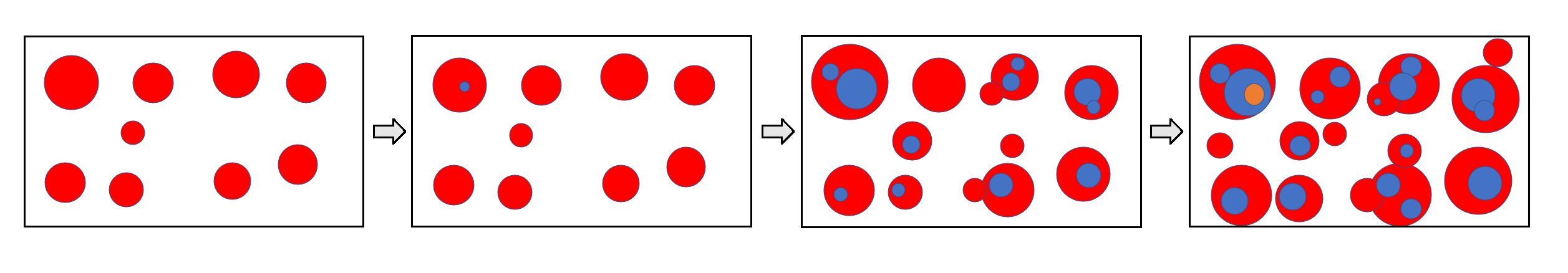}
    \\ Figure 3. Mutations of types 1, 2, 3 colored in red, blue, and orange respectively.
\end{center}

\subsection{Theorem 4: an intermediate case between Theorems 2 and 3}

Assume
$$\mu_1\gg\frac{\alpha}{N^{(d+1)/d}}.$$
We first define
\begin{align}
    \label{definel}
    l:=\max\Bigg\{j\geq 2:\mu_j\ll \frac{1}{\alpha^d\beta_{j-1}^{d+1}}\Bigg\}.
\end{align}
It follows from (\ref{increasing}) that if $\mu_j\ll 1/(\alpha^d\beta_{j-1}^{d+1})$, then $\mu_{j-1} \ll 1/(\alpha^d\beta_{j-1}^{d+1})$, which by Lemma \ref{hardestlemma} below implies that $\mu_{j-1} \ll 1/(\alpha^d\beta_{j-2}^{d+1})$.  It follows that
\begin{align}
    \label{infact}
    l=\max\Bigg\{j\geq 2:\mu_2\ll \frac{1}{\alpha^d \beta^{d+1}_{1}},\mu_3\ll \frac{1}{\alpha^d \beta^{d+1}_{2}},...,\mu_j\ll \frac{1}{\alpha^d \beta^{d+1}_{j-1}}\Bigg\}.
\end{align}
The definition of $l$ in (\ref{definel}) omits the possibility $l=1$, since $\beta_0$ is undefined. However, if we define $l = 1$ when the set over which we take the maximum in (\ref{definel}) is empty, then Theorem 4 below when $l = 1$ is the same as Theorem 2.   
On the other hand if $l \in \{k,k+1,...\}\cup \{\infty\}$, then by (\ref{infact}) we have $\mu_k\ll 1/(\alpha^d\beta_{k-1}^{d+1})$, in which case Theorem 3 applies.   
Hence we assume $l\in \{2,...,k-1\}$ and
\begin{align}
    \label{thm4assume}
    \mu_{l+1}\gg \frac{1}{\alpha^d\beta_l^{d+1}}.
\end{align}
The situation in Theorem 4 is a hybrid of Theorems 2 and 3. A mutation of type $j\in \{1,...,l-1\}$ takes a longer time to fixate in the torus than the interarrival time $\sigma_{j}-\sigma_{j-1}$. As a result, if $j\in \{2,...,l\}$, there will be many mostly nonoverlapping balls of type $j-1$ before time $\sigma_j$. Using this fact, we proceed as in Theorem 3 and find  $\lim_{N\to\infty}\PP(\sigma_l>\beta_lt)$. Next, our assumption in (\ref{thm4assume}) places us in the regime of Theorem 2; all mutations of types $l+1,...,k$ happen so quickly that for all $\epsilon>0$ we have $\mathbb{P}(\sigma_k-\sigma_l>\beta_l\epsilon)\to 0$. Then combining these two results yields the following theorem.

\begin{Theo}
\label{theorem4}
Suppose (\ref{increasing}) holds, and suppose $\mu_1\gg \alpha/N^{(d+1)/d}$.  Suppose also that $l\in \{2,...,k-1\}$ and that $\mu_{l+1}\gg 1/(\alpha^d\beta_l^{d+1})$. Then for $t>0$,
$$\mathbb{P}(\sigma_k>\beta_l t)\to \exp\Big(-\frac{\gamma_d^{l-1}(d!)^{l-1}}{((l-1)d+l)!}t^{(l-1)d+l}\Big).$$
\end{Theo}

In pictures, Theorem 4 looks like Figure 3 for mutations up to type $l$. Then once the first type $l$ mutation appears and spreads in a circle, all the subsequent mutations become nested within that circle, similar to Figure 2. 

\begin{Rmk}
{\em Theorems 1-4 cover most of the possible cases in which (\ref{increasing}) holds.  However, we assume that either $\mu_1 \ll \alpha/N^{(d+1)/d}$ or $\mu_1 \gg \alpha/N^{(d+1)/d}$.  When $\mu_1 \asymp \alpha/N^{(d+1)/d}$, we expect that at the time a type 2 mutation appears, there could be several overlapping type~1 balls whose size is comparable to the size of the torus, and we do not expect the limiting distribution of $\sigma_k$ to have a simple expression.  Consequently, we do not pursue this case here.  We note that if $\mu_1 \asymp \alpha/N^{(d+1)/d}$ and all mutation rates are equal (i.e. $\mu_i=\mu$ for all $i$), then it is proven, as a special case of Theorem 12 in \cite{fls}, that $N\mu\sigma_k$ converges in distribution to a nondegenerate random variable for every $k\geq 1$.  Likewise, we do not consider the case in which, instead 
of (\ref{thm4assume}), we have $\mu_{l+1} \asymp 1/(\alpha^d \beta_l^{d+1})$.  In this case we believe there could be several overlapping type $l$ balls at the time the first type $l+1$ mutation occurs, again preventing there from being a simple expression for the limit distribution.}
\end{Rmk}

\subsection{Distances between mutations} \label{distanceresults}

For $1\leq i<j$, define $D_{i,j}$ to be the distance in the torus between the location of the
first mutation of type $j$ and the location of the first mutation of type $i$.  Also define $D_{i+1}:=D_{i,i+1}$.

Consider the setting of Theorem 2.  We will assume a stronger version of (\ref{increasing}), which is
\begin{align}
\label{strincreasing0}
\mu_2\ll \mu_3\ll \mu_4\ll \cdots.
\end{align}
Recall that the mutations appear in nested balls as in Figure 2.
Because the first type $j+1$ mutation will therefore appear before the second type $j$ mutation with high probability, we can calculate, as in (\ref{singleball}), that
$$\mathbb{P}(\sigma_{j+1}-\sigma_j>t) \approx \exp\Big(-\frac{\mu_{j+1}\gamma_d\alpha^d}{d+1}t^{d+1}\Big).$$
It follows that if we define $\kappa_{j+1}$ as in (\ref{betakay}), then
$$\mathbb{P}(\sigma_{j+1}-\sigma_j>\kappa_{j+1}t) \approx \exp\Big(-\frac{\gamma_d}{d+1}t^{d+1}\Big).$$
With this, we can calculate the approximate probability density $f(t)$ of $(\sigma_{j+1}-\sigma_j)/\kappa_{j+1}$.
This allows us to calculate 
$$\mathbb{P}\bigg(\frac{D_{j+1}}{\alpha\kappa_{j+1}}\leq s \bigg) \approx \int_{0}^{\infty}\mathbb{P}\bigg(\frac{D_{j+1}}{\alpha\kappa_{j+1}}\leq s\Big|\frac{\sigma_{j+1}-\sigma_j}{\kappa_{j+1}}=t \bigg)f(t)dt.$$
The location of the first type $j+1$ mutation conditioned on $\sigma_{j+1}-\sigma_j=\kappa_{j+1}t$ is a uniformly random point on a $d$-dimensional ball of radius $\alpha\kappa_{j+1}t$. This allows us to calculate $\lim_{N\to\infty}\mathbb{P}(D_{j+1}\leq \alpha\kappa_{j+1}s)$. Next, because of (\ref{strincreasing0}), mutations of types $j+2,j+3,j+4,...$ appear rapidly once the first type $j+1$ appears.  This means that $D_{j+2}+\cdots+D_{j+k}$ is small relative to $D_{j+1}$, and therefore that $D_{j,k}$ has the same limiting distribution as $D_{j+1}$.
These heuristics lead to the following theorem. 

\begin{Theo}
\label{theorem5}
Suppose (\ref{strincreasing0}) holds, and suppose $\mu_1\gg\alpha/N^{(d+1)/d}$ and $\mu_2\gg (N\mu_1)^{d+1}/\alpha^d$. Suppose $1 \leq j < k$.  Then for all $s>0$,
$$\mathbb{P}\bigg(\frac{D_{j,k}}{\alpha\kappa_{j+1}}\leq s\bigg)\to\int_0^\infty \gamma_d (t \wedge s)^d \exp\Big(-\frac{\gamma_d t^{d+1}}{d+1}\Big)dt.$$
\end{Theo}

Recall the definition of $l$ in (\ref{definel}). Theorem 4 is similar to Theorem 2 in that once the first type $l$ mutation appears, all the subsequent type $l+1,l+2,...$ mutations happen quickly. Therefore, it is reasonable to expect that the type $l,l+1,l+2,...$ mutations behave similarly as the type $1,2,3,...$ mutations in Theorem \ref{theorem2}. Hence, analogous to (\ref{strincreasing0}), assume that
\begin{align}
    \label{strincreasing1}
    \mu_{l+1}\ll\mu_{l+2}\ll \mu_{l+3}\ll \cdots. 
\end{align}
We then obtain the following result.

\begin{Theo}
\label{theorem6}
Suppose (\ref{strincreasing1}) holds, and suppose $\mu_1 \gg \alpha/N^{(d+1)/d}$.  Define $l$ as in (\ref{definel}), and suppose also that $l \geq 2$ and that $\mu_{l+1}\gg 1/(\alpha^d\beta_l^{d+1})$.
Suppose $l \leq j<k$. Then for all $s>0$,
$$\mathbb{P}\bigg(\frac{D_{j,k}}{\alpha\kappa_{j+1}}\leq s\bigg)\to\int_0^{\infty} \gamma_d (t \wedge s)^d\exp\Big(-\frac{\gamma_d t^{d+1}}{d+1}\Big)dt.$$
\end{Theo}

\begin{Rmk}
{\em Theorems \ref{theorem5} and \ref{theorem6} hold in the settings of Theorems \ref{theorem2} and \ref{theorem4} respectively.  In the setting of Theorem 1,
each type $i\geq 1$ mutation fills up the entire torus before a type $i+1$ mutation occurs, and so the first type $i+1$ mutation appears at a uniformly distributed point on the torus, independently of where all previous mutations originated.  Therefore, the problem of finding the distribution of the distances between mutations becomes trivial in this case.
On the other hand, in the setting of Theorem 3, type $i$ mutations appear in small and mostly non-overlapping circles before the first type $i+1$ mutation appears.
Thus, calculating the distribution of $D_{i+1}$ requires understanding not only the total volume of the type $i$ region, but also the sizes and locations of many small type $i$ regions.  We do not pursue this case here, but we conjecture that because the first type $i+1$ mutation is likely not to appear within the type $i$ region generated by the first type $i$ mutation, the locations of the first type $i$ and the first type $i+1$ mutation should be nearly independent of each other, as in the setting of Theorem 1.}
\end{Rmk}

\section{Proofs of limit theorems for $\sigma_k$}\label{waitkproofs}
In this section, we will prove Theorems 1-4. We begin by introducing the structure of the torus $\mathcal{T}=[0,L]^d$, and will follow the notation of \cite{fls}. We define a pseudometric on the closed interval $[0,L]$ by
$$d_L(x,y):=\min\{|x-y|,L-|x-y|\}.$$
The $d$-dimensional torus of side length $L$ will be denoted by $\mathcal{T}=[0,L]^d$. For $x=(x^1,...,x^d)\in \mathcal{T}$ and $y=(y^1,...,y^d)\in \mathcal{T}$ we define a pseudometric by
$$|x-y|:=\sqrt{\sum_{i=1}^{d}d_L(x^i,y^i)^2}.$$
The torus should be viewed as $\mathcal{T}$ modulo the equivalence relation $x\sim y$ iff $|x-y|=0$, or more simply $\mathcal{T}=(\mathbb{R}\text{ mod }L)^d$. However, we will continue to write $\mathcal{T}=[0,L]^d$, keeping in mind that certain points are considered to be the same via the equivalence relation defined above. It will be useful to observe the following:
\begin{enumerate}
    \item We have $d_L(x,y)\leq L/2$ for all $x,y\in [0,L]$. As a result, the distance between any two points $x,y\in \mathcal{T}$ is at most
    \begin{align}
        \label{maximaldistance}
        \sup\{|x-y|:x,y\in \mathcal{T}\}= \sqrt{\sum_{i=1}^{d}\Big(\frac{L}{2}\Big)^2}=\frac{\sqrt{d}L}{2}.
    \end{align}
    \item Therefore, once a mutation of type $j$ appears, the entire torus will become type $j$ in time
    \begin{align}
        \label{fixation time}
        \frac{\text{maximal distance between any }x,y\in\mathcal{T}}{\text{rate of mutation spread per unit time}}=\frac{\sqrt{d}L}{2\alpha}.
    \end{align}
\end{enumerate}
We use $|A|$ to denote the Lebesgue measure of some subset $A$ of $\mathcal{T}$ or $\mathcal{T}\times [0,\infty)$, so that $N=L^d=|\mathcal{T}|$ is the torus volume. Each $x\in \mathcal{T}$ at time $t$ has a type $k\in \{0,1,2,...\}$, which we denote by $T(x,t)$, corresponding to the number of mutations the site has acquired. The set of type $i$ sites is defined by $$\chi_i(t):=\{x\in \mathcal{T}:T(x,t)=i\}$$
The set of points whose type is at least $i$ is defined by
$$\psi_i(t):=\{x\in \mathcal{T}:T(x,t)\geq i\}=\bigcup_{j=i}^{\infty}\chi_j(t).$$
At time $t$, we denote the total volume of type $i$ sites by $X_i(t):=|\chi_i(t)|$, and the total volume of sites with type at least $i$ by $Y_i(t):=|\psi_i(t)|$.  The first time a type $k$ mutation appears in the torus can be expressed as $\sigma_k=\inf\{t>0:Y_k(t)>0\}$. 

Still following \cite{fls}, we now explicitly describe the construction of the process which gives rise to mutations in the torus. We model mutations as random space-time points $(x,t)\in \mathcal{T}\times [0,\infty)$. Let $(\Pi_k)_{k=1}^{\infty}$ be a sequence of independent Poisson point processes on $\mathcal{T}\times [0,\infty)$, where $\Pi_k$ has intensity $\mu_k$. That is, for any space-time region $A\subseteq  \mathcal{T}\times [0,\infty)$, the probability that $A$ contains $j$ points of type $k$ is
$$e^{-\mu_k|A|}\frac{(\mu_k|A|)^j}{j!}.$$
Each $(x,t)\in \Pi_k$ is a space-time point at which $x\in \mathcal{T}$ can acquire a $k$th mutation at time $t$. We say that $x$ mutates to type $k$ at time $t$ precisely when $x\in \chi_{k-1}(t)$ and $(x,t)\in \Pi_k$. Once an individual obtains a type $k$ mutation, it spreads the type $k$ mutations outward in a ball at rate $\alpha$ per unit time.


\subsection{Proof of Theorem 1}
In the setting of Theorem 1, once the first mutation appears, with high probability it spreads to the entire torus before another mutation appears.  The proof of Theorem 1 uses Theorem 1 of \cite{fls}, which we restate  below as Theorem A.  Theorem 1 is very similar to Theorem A when $j=1$.  However, Theorem A requires $\mu_j \ll \alpha/N^{(d+1)/d}$ for all $j \in \{1, \dots, k\}$, whereas Theorem 1 requires this condition only for $j = 1$.  This is why Theorem 1 can not be deduced directly from Theorem A, even though the proofs of the results are essentially the same.
\\~\\
\textbf{Theorem A.} \textit{Suppose $\mu_i\ll \alpha/N^{(d+1)/d}$ for $i\in \{1,...,k-1\}$. Suppose there exists $j\in \{1,...,k\}$ such that $\mu_j\ll \alpha/N^{(d+1)/d}$ and 
$$\frac{\mu_i}{\mu_j}\to c_i\in (0,\infty]\text{ for all }i\in \{1,...,k\}.$$
Let $W_1,...,W_k$ be independent random variables such that $W_i$ has an exponential distribution with rate parameter $c_i$ if $c_i<\infty$ and $W_i=0$ if $c_i=\infty$. Then
$$N\mu_j\sigma_k\Rightarrow W_1+\cdots+W_k.$$}

\noindent \textit{Proof of Theorem 1.} Let $r:=\max\{j\in \{1,...,k\}:\mu_j\lesssim \mu _1\}$.
Then
$ \mu_j\ll \alpha/N^{(d+1)/d}$ for all $j \in \{1, \dots, r\}$. By Theorem A, we have
$$N\mu_1\sigma_r\Rightarrow W_1+\cdots+ W_r.$$
If $r=k$, then the conclusion follows. Otherwise, $r\leq k-1$, and by the maximality of $r$ and equation (\ref{increasing}), we have $\mu_l/\mu_1 \rightarrow \infty$ for all $l\in \{r+1,...,k\}$.
Then the result follows if we show 
$$N\mu_1(\sigma_k-\sigma_r)\to_p 0.$$
We have
\begin{equation}\label{teles}
0\leq N\mu_1(\sigma_k-\sigma_r) = N\mu_1\sum_{j=r}^{k-1}(\sigma_{j+1}-\sigma_j).
\end{equation}
We will find an upper bound for the right-hand side of (\ref{teles}). For $i\geq 1$, let $$t_i=\inf\{t>0:Y_i(t)=N\}$$ be the first time which every point in $\mathcal{T}$ is of at least type $i$. Define $\hat{t}_i:=t_i-\sigma_i$, which is the time elapsed between $\sigma_i$ and when mutations of type $i$ fixate in the torus. Also define $\hat{\sigma}_i=\inf\{t>0:\Pi_i\cap (\mathcal{T}\times [t_{i-1},t])\neq \varnothing\}$, which is the first time there is a potential type $i$ mutation after $t_{i-1}$. Observe that because we always have $\sigma_i \leq \hat{\sigma}_i$,
$$\sigma_{j+1}-\sigma_j \leq \hat{\sigma}_{j+1}-\sigma_j = \hat{\sigma}_{j+1}-\sigma_j +t_j-t_j = \hat{t}_j + (\hat{\sigma}_{j+1}-t_j).$$
Also observe that by (\ref{fixation time}), we have
$\hat{t}_j\leq \sqrt{d}N^{1/d}/(2\alpha).$
Consequently,
the right-hand side of (\ref{teles}) has the upper bound
\begin{align*}
    &  N\mu_1\Big(\sum_{j=r}^{k-1}\hat{t}_j+\sum_{j=r}^{k-1}(\hat{\sigma}_{j+1}-t_j)\Big)  \leq N\mu_1(k-r)\frac{\sqrt{d}N^{1/d}}{2\alpha} + N\mu_1\sum_{j=r}^{k-1}(\hat{\sigma}_{j+1}-t_j). 
\end{align*}
The result follows if the right-hand side of the above expression converges to $0$ in probability. The first term tends to zero because $\mu_1\ll \alpha/N^{(d+1)/d}$. The second term tends to zero because $\hat{\sigma}_{j+1}-t_j\sim \text{Exponential}(N\mu_{j+1})$, so
$N\mu_1(\hat{\sigma}_{j+1}-t_j)\sim \text{Exponential}(\mu_{j+1}/\mu_1)\to_p 0$. \QED{}

\subsection{Proof of Theorem 2}

\begin{Lem}
\label{computeexpectation}
Let $t_N$ be a random time that is $\sigma(\Pi_1,...,\Pi_{j-1})$-measurable and satisfies $t_N\geq \sigma_{j-1}$. Then
$$\mathbb{P}(\sigma_j>t_N)=\mathbb{E}\Bigg[\exp\Big(-\int_{\sigma_{j-1}}^{t_N}\mu_jY_{j-1}(s)ds\Big)\Bigg].$$
\end{Lem}

\begin{proof}
Write $\mathcal{G}:=\sigma(\Pi_1,...,\Pi_{j-1})$. Define the set 
$$A:=\{(x,r)\in \psi_{j-1}(r)\times [\sigma_{j-1},t_N]\},$$
and note that the Lebesgue measure of this set, which we denote by $|A|$, is a ${\cal G}$-measurable random variable.  The event $\{\sigma_{j}>t_N\}$ occurs precisely when
$\Pi_j\cap A =\varnothing.$  Let $X$ be the number of points of $\Pi_j$ in the set $A$.  Because $\Pi_j$ is independent of $\Pi_1, \dots, \Pi_{j-1}$, the conditional distribution of $X$ given ${\cal G}$ is Poisson$(\mu_j|A|)$.  Therefore,
$$\mathbb{P}(\sigma_j>t_N|\mathcal{G})=\mathbb{P}(X=0|\mathcal{G})= \exp(-\mu_j |A|) = \exp\Big(-\int_{\sigma_{j-1}}^{t_N}\mu_jY_{j-1}(s)ds\Big).$$
Taking expectations of both sides finishes the proof.
\end{proof}

\noindent \textit{Proof of Theorem 2.} Write $N\mu_1\sigma_k$ as a telescoping sum
$$N\mu_1\sigma_k=N\mu_1\sigma_1+\sum_{j=2}^{k}N\mu_1(\sigma_j-\sigma_{j-1}).$$
We have $N\mu_1\sigma_1\sim \text{Exponential}(1)$. Hence it suffices to show that for each $j\geq 2$, the random variable $N\mu_1(\sigma_j-\sigma_{j-1})$ converges in probability to zero. Let $t>0$. Then by Lemma \ref{computeexpectation},
\begin{align}
    \mathbb{P}(N\mu_1(\sigma_j-\sigma_{j-1})> t) & = \mathbb{P}\Big(\sigma_j> \frac{t}{N\mu_1}+\sigma_{j-1}\Big) \nonumber
    \\ & = \mathbb{E}\Bigg[\exp\Big(-\int_{\sigma_{j-1}}^{t/(N\mu_1)+\sigma_{j-1}}\mu_jY_{j-1}(s)ds\Big)\Bigg]. \nonumber
\end{align}
We want to show that the term on the right-hand side tends to zero. By the dominated convergence theorem, it suffices to show that as $N\to \infty$,
$$\int_{\sigma_{j-1}}^{t/(N\mu_1)+\sigma_{j-1}}\mu_jY_{j-1}(s)ds\to \infty.$$
Notice that because $\mu_1\gg \alpha/N^{(d+1)/d}$, for all sufficiently large $N$ we have $t/(N\mu_1)\leq N^{1/d}/(2\alpha)$.  Therefore, at time $\sigma_{j-1} + t/(N \mu_1)$, there is a ball of type $j-1$ mutations of radius $\alpha(t - \sigma_{j-1})$ which has not yet begun to wrap around the torus and overlap itself.
Hence, we have $Y_{j-1}(s)\geq \gamma_d\alpha^d(s-\sigma_{j-1})^d$ for $s\in [\sigma_{j-1},\sigma_{j-1}+t/(N\mu_1)]$, and therefore
\begin{align*}
    \int_{\sigma_{j-1}}^{t/(N\mu_1)+\sigma_{j-1}}\mu_jY_{j-1}(s)ds& \geq \int_{\sigma_{j-1}}^{t/(N\mu_1)+\sigma_{j-1}}\mu_j\gamma_d\alpha^d (s-\sigma_{j-1})^d ds
    \\ & = \int_{0}^{t/(N\mu_1)}\mu_j\gamma_d \alpha^d u^d du
    \\ & = \frac{\mu_j\gamma_d \alpha^d}{d+1}\Big(\frac{t}{N\mu_1}\Big)^{d+1}.
\end{align*}
It remains to show
$$\frac{\mu_j\gamma_d\alpha^d}{d+1}\Big(\frac{t}{N\mu_1}\Big)^{d+1}\to \infty\text{ as }N\to \infty.$$
For the above to hold, it suffices to have $ \mu_j\gg (N\mu_1)^{d+1}/\alpha^d$, which holds due to the second assumption in the theorem and equation (\ref{increasing}). This finishes the proof. \QED{} 

\subsection{Proof of Theorem 3}

\noindent We recall the definition of $\beta_k$ as in (\ref{betakay}) of Section 2. In the setting of Theorem 3, $\beta_k$ is the order of magnitude of the time it takes for the $k$th mutation to appear. 

Much of the proof of Theorem 3 will rely on Lemma 9 of \cite{fls}, which approximates a monotone stochastic process by a deterministic function under a certain time-scaling. In order to apply this lemma, it is important to ensure that $Y_k(t)$ is well-approximated by its expectation, which is Lemma 8 of the same paper.

Before proving Theorem 3, we state several lemmas, some of which are from \cite{fls}. First, we need to ensure that the last assumption  $\mu_k\alpha^d\beta_{k-1}^{d+1}\to 0$ in Theorem 3 implies $\mu_k\alpha^d\beta_{k}^{d+1}\to 0$, so that we are able to use part 2 of Lemma \ref{lemmafiveeight} to approximate $Y_{k-1}(\beta_k t)$ by its expectation. 
\begin{Lem}
\label{hardestlemma}
For $k\geq 2$, we have $\mu_k\ll 1/(\alpha^d\beta_k^{d+1})$ if and only if $\mu_k\ll 1/(\alpha^d\beta_{k-1}^{d+1})$.
\end{Lem}
\begin{proof}
    By using the definition of $\beta_k$ from (\ref{betakay}), we get
\begin{align*}
     \mu_k \ll \frac{1}{\alpha^d\beta_k^{d+1}} & \iff \mu_k^{(k-1)d + k} \ll \frac{1}{\alpha^{d[(k-1)d + k]}} \Big(N \alpha^{(k-1)d} \prod_{i=1}^k \mu_k\Big)^{d+1}
     \\ & \iff \mu_k^{(k-2)d+(k-1)}\ll \frac{1}{\alpha^d}N^{d+1}\Big(\prod_{i=1}^{k-1}\mu_i\Big)^{d+1}
    \\ & \iff \mu_k^{(k-2)d+(k-1)}\ll \frac{\alpha^{d(d+1)(k-2)}}{\alpha^{d[(k-2)d+(k-1)]}}N^{d+1}\Big(\prod_{i=1}^{k-1}\mu_i\Big)^{d+1}
    \\ & \iff \mu_k^{(k-2)d+(k-1)}\ll \frac{1}{\alpha^{d[(k-2)d+(k-1)]}}\Big(N\alpha^{(k-2)d}\prod_{i=1}^{k-1}\mu_i\Big)^{d+1}
    \\ & \iff \mu_k \ll \frac{1}{\alpha^d\beta_{k-1}^{d+1}}
\end{align*}
as claimed.
\end{proof}
We also need Lemma 9 from \cite{fls}, which is restated as Lemma \ref{approxbydeterministic} below.  This lemma gives necessary conditions under which a monotone stochastic process is well-approximated by a deterministic function.

\begin{Lem}
\label{approxbydeterministic}
Suppose, for all positive integers $N$, $(Y_N(t),t\geq 0)$ is a nondecreasing stochastic process such that $\mathbb{E}[Y_N(t)] < \infty$ for each $t > 0$.  Assume there exist sequences of positive numbers $(\nu_N)_{N=1}^{\infty}$ and $(s_N)_{N=1}^{\infty}$ and a continuous nondecreasing function $g>0$ such that for each fixed $t>0$ and $\epsilon>0$, we have
\begin{equation}\label{lemmanine1}
    \lim_{N\to\infty}\mathbb{P}(|Y_N(s_Nt)-\mathbb{E}[Y_N(s_Nt)]|>\epsilon\mathbb{E}[Y_N(s_Nt)])=0
\end{equation}
and 
\begin{equation}\label{lemmanine2}
\lim_{N\to\infty}\frac{1}{\nu_N}\mathbb{E}[Y_N(s_Nt)]=g(t).
\end{equation}
Then for all $\epsilon>0$ and $\delta>0$, we have
$$\lim_{N\to\infty}\mathbb{P}(\nu_Ng(t)(1-\epsilon)\leq Y_N(s_Nt)\leq \nu_N g(t)(1+\epsilon)\enspace for\enspace all \enspace t\in [\delta,\delta^{-1}])=1.$$
\end{Lem}

Next, we state a criterion which guarantees that for fixed $t>0$, the probability $\mathbb{P}(\sigma_k>\beta_kt)$ converges to a deterministic function as $N\to\infty$.

\begin{Lem}
\label{theorem10tool}
For a continuous nonnegative function $g$, a sequence $(\nu_N)_{N=1}^{\infty}$ of positive real numbers, and $\delta,\epsilon>0$, define the event
$$B_N^{k-1}(\delta,\epsilon,g,\nu_N)=\{g(u)(1-\epsilon)\nu_N\leq Y_{k-1}(\beta_ku)\leq g(u)(1+\epsilon)\nu_N,\enspace for \enspace all \enspace u\in [\delta,\delta^{-1}] \}.$$
If $(\nu_N)_{N=1}^{\infty}$ and $g$ are chosen such that $\displaystyle \lim_{N\to\infty}\nu_N\beta_k\mu_k$ exists and $\displaystyle \lim_{N\to\infty}\mathbb{P}(B_N^{k-1}(\delta,\epsilon,g,\nu_N))=1$, then
$$\lim_{N\to\infty}\mathbb{P}(\sigma_k>\beta_kt)=\lim_{N\to\infty}\exp\Big(-\nu_N\beta_k\mu_k\int_{0}^{t}g(u)du\Big).$$
\end{Lem}

\begin{proof}
Suppose $\delta \leq t \leq \delta^{-1}$.  We reason as in the proof of Theorem 10 from \cite{fls}.  The upper and lower bounds from (26) and (27) in \cite{fls} are
$$\mathbb{P}(\sigma_k>\beta_k t)\leq \exp\Big(-\mu_k\beta_k\nu_N(1-\epsilon)\int_{\delta}^{t}g(u)du\Big)+\mathbb{P}(B_N^{k-1}(\delta,\epsilon,g,\nu_N)^c)$$ and
$$\mathbb{P}(\sigma_k>\beta_kt)\geq \mathbb{P}(B_N^{k-1}(\delta,\epsilon,g,\nu_N))\exp\Big(-\nu_N(1+\epsilon)\beta_k\mu_k\int_{\delta}^{t}g(u)du\Big)-\frac{\gamma_d^{k-1}(d!)^{k-1}}{(d(k-1)+k)!}\delta^{d(k-1)+k}.$$
Taking $N\to\infty$ and then $\epsilon,\delta\to 0$, we get the desired result.
\end{proof}

We also need to approximate the expected volume of type $k$ or higher regions, $\mathbb{E}[Y_k(t)]$, with a deterministic function, as well as making sure that $Y_k(t)$ is well-approximated by its expectation.  Lemma \ref{lemmafiveeight} below is a restatement of Lemmas 5 and 8 in \cite{fls}.  It is important to note that for this result, the time $t$ may depend on $N$.

\begin{Lem}
\label{lemmafiveeight}
Fix a positive integer $k$. Suppose $\mu_j\alpha^d t^{d+1}\to 0$ for all $j\in \{1,...,k\}$. Also suppose $t \leq N^{1/d}/(2\alpha)$. Then 
\begin{enumerate}
    \item Setting $\displaystyle v_k(t):=\frac{\gamma_d^k(d!)^k}{(k(d+1))!}\Big(\prod_{i=1}^{k}\mu_i\Big)N\alpha^{kd}t^{k(d+1)}$, we have $\mathbb{E}[Y_k(t)]\sim v_k(t).$
    \item If in addition we assume $\displaystyle \Big(\prod_{i=1}^{k}\mu_i\Big)N\alpha^{(k-1)d}t^{(k-1)d+k}\to\infty$, then for all $\epsilon>0$,
    $$\lim_{N\to\infty}\mathbb{P}((1-\epsilon)\mathbb{E}[Y_k(t)]\leq Y_k(t)\leq (1+\epsilon)\mathbb{E}[Y_k(t)])=1.$$
    \end{enumerate}
\end{Lem}

\begin{Rmk}
{\em Lemma 5 in \cite{fls} omits the necessary hypothesis $t \leq N^{1/d}/(2\alpha)$. This hypothesis ensures that a growing ball of mutations can not begin to wrap around the torus and overlap itself before time $t$, which is needed for the formula for $\mathbb{E}[\Lambda_{k-1}(t)]$ in equation (15) of \cite{fls} to be exact.  This equation is used in the proof of Lemma 5 of \cite{fls}.
Note that the hypothesis $t \leq N^{1/d}/(2\alpha)$ is also needed for Lemma 8 in \cite{fls}, because its proof uses Lemma 5 in \cite{fls}.  However, because it is easily verified that this hypothesis is satisfied in \cite{fls} whenever these lemmas are used, all of the main results in \cite{fls} are correct without additional hypotheses.}
\end{Rmk} 

The next lemma states that if $\mu_1\gg \alpha/N^{(d+1)/d}$, then $\beta_l$ is much smaller than the time it takes for a mutation to spread to the entire torus.

\begin{Lem}
\label{lemmathm3}
Suppose $\mu_1\gg \alpha/N^{(d+1)/d}$ and (\ref{increasing}) holds.  Then $\beta_l \ll N^{1/d}/\alpha$ for any $l\in \mathbb{N}$.
\end{Lem}

\begin{proof}
By (\ref{increasing}), we have $\mu_1,...,\mu_l\gg \alpha/N^{(d+1)/d}$. Thus
$$\prod_{i=1}^{l}\mu_i\gg \frac{\alpha^l}{N^{l(1+1/d)}}.$$
On the other hand by simplifying,
\begin{align*}
    \beta_l\ll \frac{N^{1/d}}{\alpha} & \iff N\alpha^{(l-1)d}\prod_{i=1}^{l}\mu_i\gg \Big(\frac{\alpha}{N^{1/d}}\Big)^{(l-1)d+l} \iff \prod_{i=1}^{l}\mu_i \gg \frac{\alpha^{l}}{N^{l(1+1/d)}}.
\end{align*}
This proves the lemma.
\end{proof}

\noindent \textit{Proof of Theorem 3.} In view of Lemma \ref{theorem10tool}, 
we will choose $(\nu_N)_{N=1}^{\infty}$ and a continuous nonnegative function $g_k$ such that $\lim_{N\to \infty}\nu_N\beta_k\mu_k$ exists and $\mathbb{P}(B_N^{k-1}(\delta,\epsilon,g_k,\nu_N))\to 1$ as $N\to \infty$. We set $\nu_N=1/(\beta_k \mu_k)$, and as in the proof of Theorem 10 in \cite{fls}, set
$$g_k(t):=\frac{\gamma_d^{k-1}(d!)^{k-1}t^{(k-1)(d+1)}}{((k-1)(d+1))!}.$$
A lengthy calculation shows that $\beta_k\mu_kv_{k-1}(\beta_kt)=g_k(t)$. On the other hand, by the last assumption in the theorem, we have $\mu_k\alpha^d\beta_{k-1}^{d+1}\to 0$. By Lemma \ref{hardestlemma}, this is equivalent to $\mu_k\alpha^d\beta_{k}^{d+1}\to 0$.  Because of (\ref{increasing}), this implies that 
$$\mu_j\alpha^d(\beta_k t)^{d+1}\to 0$$
for all $j\in \{1,...,k\}$. Also, because of Lemma \ref{lemmathm3}, we have $\beta_k\ll N^{1/d}/(2\alpha)$. Hence the hypotheses of Lemma \ref{lemmafiveeight} are satisfied, and by the first result in Lemma \ref{lemmafiveeight} applied to $k-1$, it follows that $v_{k-1}(\beta_k t)\sim\mathbb{E}[Y_{k-1}(\beta_kt)]$, which implies $$\lim_{N\to \infty}\beta_k\mu_k \mathbb{E}[Y_{k-1}(\beta_k t)]=\lim_{N\to\infty }\beta_k\mu_kv_{k-1}(\beta_kt)=g_k(t).$$
Hence, (\ref{lemmanine2}) of Lemma \ref{approxbydeterministic} is satisfied.
A direct calculation gives
$$\Big(\prod_{i=1}^{k-1}\mu_i\Big)N\alpha^{(k-2)d}\beta_k^{(k-2)d+k-1}=\frac{1}{\mu_k\alpha^d \beta_k^{d+1}}\to \infty,$$
which by the second result of Lemma \ref{lemmafiveeight} is sufficient to give (\ref{lemmanine1}).
Therefore, Lemma \ref{approxbydeterministic} guarantees that $\mathbb{P}(B_N^{k-1}(\delta,\epsilon,g_k,\nu_N))\to 1$ as $N\to\infty$. Then, Lemma \ref{theorem10tool} gives us
\begin{align*}
    \lim_{N\to\infty}\mathbb{P}(\sigma_k>\beta_k t) & = \lim_{N\to\infty } \exp\Big(-\nu_N\beta_k\mu_k\int_{0}^{t}g_k(u)du\Big)
    \\ & = \exp\Big(-\int_{0}^{t}\frac{\gamma_d^{k-1}(d!)^{k-1}u^{(k-1)(d+1)}}{((k-1)(d+1))!}du\Big)
    \\ & = \exp\Big(-\frac{\gamma_d^{k-1}(d!)^{k-1}}{(d(k-1)+k)!}t^{d(k-1)+k}\Big)
\end{align*}
finishing the proof. \QED{}

\subsection{Proof of Theorem 4}
Now we turn to proving Theorem 4, which is a hybrid of Theorems 2 and 3. In particular, we assume that there is some $l\in\mathbb{N}$ such that the mutation rates $\mu_1,\mu_2,...,\mu_l$ fall under the regime of Theorem 3, and all subsequent mutation rates $\mu_{l+1},...,\mu_k$ are large enough so that all mutations after the first type $l$ mutation occur quickly, as in Theorem 2.
\\~\\
\noindent \textit{Proof of Theorem 4.}
For ease of notation, set, for $j\in\mathbb{N}$ and $t\geq 0$,
\begin{align*}
    f_j(t):=\exp\Big(-\frac{\gamma_d^{j-1}(d!)^{j-1}t^{d(j-1)+j}}{(d(j-1)+j)!}\Big) 
\end{align*}
For $\epsilon>0$, we have the inequalities
$$\mathbb{P}(\sigma_l>\beta_l t)\leq \mathbb{P}(\sigma_k>\beta_l t)\leq \PP(\sigma_l>\beta_l(t-\epsilon))+\PP(\sigma_k-\sigma_l>\beta_l\epsilon).$$
Taking $N\to \infty$ and using Theorem 3 (noting that $l\geq 2$), we have
$$f_l(t)\leq \lim_{N\to\infty} \PP(\sigma_k>\beta_l t)\leq f_l(t-\epsilon )+\lim_{N\to \infty}\PP(\sigma_k-\sigma_l>\beta_l\epsilon).$$
Since $f_l$ is continuous, the result follows (by taking $\epsilon\to 0$) once we show that for each fixed $\epsilon>0$
\begin{equation}\label{sigmaconvprob}
\lim_{N\to\infty}\PP(\sigma_k-\sigma_l>\beta_l\epsilon)=0.
\end{equation}
Notice that because
$$\{\sigma_k-\sigma_l>\beta_l\epsilon\}\subseteq \bigcup_{j=l}^{k-1}\Big\{\sigma_{j+1}-\sigma_j>\frac{\beta_l\epsilon}{k-l}\Big\}$$
it suffices to show that for all $j\in \{l,...,k-1\}$ we have
$$\PP(\sigma_{j+1}-\sigma_j>\beta_l\epsilon)\to  0.$$
By Lemma \ref{computeexpectation}, we have 
\begin{align*}
    \PP(\sigma_{j+1}-\sigma_j>\beta_l\epsilon) & = \mathbb{E}\Bigg[\exp\Big(-\int_{\sigma_j}^{\beta_l\epsilon+\sigma_j}\mu_{j+1}Y_j(s)ds\Big)\Bigg].
\end{align*}
Hence, by the dominated convergence theorem, to show that
$\PP(\sigma_{j+1}-\sigma_j>\beta_l\epsilon)\to 0$, it suffices to show that
$$\int_{\sigma_j}^{\beta_l\epsilon+\sigma_j}\mu_{j+1}Y_j(s)ds\to \infty \quad \text{a.s.}$$
By Lemma \ref{lemmathm3} we have $\beta_l\ll N^{1/d}/\alpha$, so $ \beta_l\epsilon \leq N^{1/d}/(2\alpha)$ for large enough $N$. That is, $\beta_l\epsilon$ does not exceed the time it takes for a mutation to wrap around the torus. Hence, we have the lower bound $Y_j(s)\geq \gamma_d\alpha^d (s-\sigma_j)^d$ for $s\in [\sigma_j,\sigma_j+\beta_l\epsilon]$, and
\begin{align}\label{newmuY}
    \int_{\sigma_j}^{\beta_l\epsilon+\sigma_j}\mu_{j+1}Y_j(s)ds & \geq \int_{\sigma_j}^{\beta_l\epsilon+\sigma_j}\mu_{j+1}\gamma_d\alpha^d(s-\sigma_j)^d = \frac{\mu_{j+1}\gamma_d \alpha^d}{d+1}(\beta_l \epsilon)^{d+1}.
\end{align}
By the second assumption in the theorem, we have $\mu_{l+1}\gg 1/(\alpha^d\beta_l^{d+1})$. Because of (\ref{increasing}), we have $\mu_{j+1}\gg 1/(\alpha^d\beta_l^{d+1})$. It follows that the right-hand side of (\ref{newmuY}) tends to infinity as $N\to \infty$, which completes the proof. \QED{}

\section{Proofs of limit theorems for distances between mutations}\label{distanceproof}

\noindent For each $j\geq 1$, we define $\sigma_j^{(2)}$ to be the time at which the second type $j$ mutation occurs, i.e.
$$\sigma_j^{(2)}:=\inf\{t>\sigma_j: (x, t) \in \Pi_j \mbox{ and } x \in \psi_{j-1}(t) \mbox{ for some }x \in \mathcal{T} \}.$$
Note that $\sigma_j^{(2)}$ is defined to be the first time, after time $\sigma_j$, that a point of $\Pi_j$ lands in a region of type $j-1$ or higher.  If this point lands in a region of type $j-1$, then a new type $j$ ball will begin to grow.  If this point lands in a region of type $j$ or higher, then the evolution of the process is unaffected.
Also recall from (\ref{betakay}) that
$$\kappa_j:=(\mu_j\alpha^d)^{-1/(d+1)}.$$

\subsection{Proof of Theorem 5}

We begin with an upper bound for the time between first mutations of consecutive types.

\begin{Lem}
\label{nestedballs0}
Assume $\mu_1\gg \alpha/N^{(d+1)/d}$. Suppose $i$ and $j$ are positive integers.  Then for each fixed $t>0$, we have, for all sufficiently large $N$,
$$\mathbb{P}(\sigma_{j+1}-\sigma_j>\kappa_{i}t)\leq \exp\Big(-\frac{\gamma_d}{d+1}\cdot \frac{\mu_{j+1}}{\mu_i}t^{d+1}\Big).$$
\end{Lem}

\begin{proof}
Using Lemma \ref{computeexpectation}, we have
$$\mathbb{P}(\sigma_{j+1}-\sigma_j>\kappa_i t)=\mathbb{E}\bigg[\exp\Big(-\int_{\sigma_j}^{\sigma_j+\kappa_{i}t}\mu_{j+1}Y_j(s)ds\Big)\bigg].$$
Because of $\mu_1\gg \alpha/N^{(d+1)/d}$ and (\ref{increasing}), for all sufficiently large $N$ we have $\kappa_{i}t<L/(2\alpha)$. Thus, $Y_j(s)\geq \gamma_d \alpha^d(s-\sigma_j)^d$ for $s\in [\sigma_j,\sigma_j+\kappa_{i}t]$. Then
\begin{align}
    \mathbb{P}(\sigma_{j+1}-\sigma_j>\kappa_i t)& \leq \exp\Big(-\int_{\sigma_j}^{\sigma_j+\kappa_{i}t}\mu_{j+1}\gamma_d\alpha^d(s-\sigma_j)^{d}ds\Big) \nonumber
    \\ & = \exp\Big(-\int_{0}^{\kappa_{i}t}\mu_{j+1}\gamma_d(\alpha s)^d ds\Big) \nonumber 
    \\  & = \exp\Big(-\frac{\mu_{j+1}\gamma_d\alpha^d}{d+1}(\kappa_{i}t)^{d+1}\Big) \nonumber
    \\ & = \exp\Big(-\frac{\gamma_d}{d+1}\cdot \frac{\mu_{j+1}}{\mu_i}t^{d+1}\Big). \nonumber
\end{align}
This finishes the proof.
\end{proof}

By Lemma \ref{nestedballs0}, when $\mu_1\gg \alpha/N^{(d+1)/d}$ the interarrival time $\sigma_{j}-\sigma_{j-1}$ is at most the same order of magnitude as $\kappa_{j}$. Lemma \ref{nestedballs1} further shows that if in addition $\mu_j\ll \mu_{j+1}\ll \mu_{j+2}\ll \cdots$, then mutations of type $m > j$ appear on an even faster time scale. 

\begin{Lem}
\label{nestedballs1}
Assume $\mu_1\gg \alpha/N^{(d+1)/d}$.  Suppose $j$ is a positive integer, and $\mu_j\ll \mu_{j+1} \ll \mu_{j+2} \ll \cdots$.  Then $(\sigma_m-\sigma_j)/\kappa_j\to_p 0$ for every $m>j$.
\end{Lem}

\begin{proof}
Using (\ref{increasing}), we get
$$\frac{\sigma_m-\sigma_j}{\kappa_j}=\sum_{i=j}^{m-1}\frac{\sigma_{i+1}-\sigma_i}{\kappa_j} \lesssim 
\sum_{i=j}^{m-1}\frac{\sigma_{i+1}-\sigma_i}{\kappa_i},$$
so it suffices to show that $(\sigma_{i+1} - \sigma_i)/\kappa_i \rightarrow_p 0$ for all $i \in \{j, j+1, \dots, m-1\}$.
Let $\epsilon>0$. Using Lemma \ref{nestedballs0}, we have, for all sufficiently large $N$,
$$\mathbb{P}(\sigma_{i+1}-\sigma_i>\kappa_i\epsilon)\leq \exp\Big(-\frac{\gamma_d}{d+1}\cdot \frac{\mu_{i+1}}{\mu_{i}}\epsilon^{d+1}\Big).$$
Then $\mathbb{P}(\sigma_{j+1}-\sigma_j>\kappa_j\epsilon)\to 0$ because $\mu_{j+1}/\mu_j\to \infty$, as desired.
\end{proof}

Next, we want to show that the balls from different mutation types become nested, as in Figure 2. That is, for any $i\geq 1$ and $j>i$, we have $\mathbb{P}(\sigma_j<\sigma_i^{(2)})\to 1$, meaning that a type $j$ mutation appears before a second type $i$ mutation can appear. We first prove the case when $i=1$ in Lemma \ref{nestedballs2} below, assuming the same hypotheses as in Theorem \ref{theorem2}.

\begin{Lem}
\label{nestedballs2}
Suppose (\ref{increasing}) holds, and suppose $\mu_1 \gg \alpha/N^{(d+1)/d}$ and $\mu_2\gg (N\mu_1)^{d+1}/\alpha^d$. Then
\begin{enumerate}
    \item For all $t>0$, we have $\mathbb{P}(\sigma_2+\kappa_2t<\sigma_1^{(2)})\to 1$.
    \item For every $j\geq 2$, we have $\mathbb{P}(\sigma_j<\sigma_1^{(2)})\to 1$.
\end{enumerate}
\end{Lem}

\begin{proof} To prove the first statement, let $\epsilon>0$.  It was shown in the proof of Theorem 2 that  $N\mu_1(\sigma_2-\sigma_1)\to_p 0$. Also, the assumption $\mu_2\gg (N\mu_1)^{d+1}/\alpha^d$ implies $N\mu_1\kappa_2 t\to 0$. Thus, we have $N\mu_1(\sigma_2-\sigma_1)+N\mu_1\kappa_2t\to_p 0.$ Therefore, for all sufficiently large $N$,
$$\mathbb{P}(N\mu_1(\sigma_2-\sigma_1)+N\mu_1\kappa_2 t<\epsilon)> 1-\epsilon.$$
On the other hand, because $N\mu_1(\sigma_1^{(2)}-\sigma_1)$ has an $\text{Exponential}(1)$ distribution, we have
$$\mathbb{P}(N\mu_1(\sigma_1^{(2)}-\sigma_1)>\epsilon)=e^{-\epsilon}>1-\epsilon.$$
Combining the above, we find
\begin{align*}
    \mathbb{P}(\sigma_2+\kappa_2 t<\sigma_1^{(2)}) \geq \mathbb{P}(N\mu_1(\sigma_2-\sigma_1)+N\mu_1\kappa_2 t<\epsilon <N\mu_1(\sigma_1^{(2)}-\sigma_1)) > 1-2\epsilon.
\end{align*}
This proves the first statement. For the second statement, by the proof of Theorem $2$ again, we have $N\mu_1(\sigma_i-\sigma_{i-1})\to_p 0$ for every $i\geq 2$. Thus,
$$N\mu_1(\sigma_j-\sigma_1)=\sum_{i=2}^{j}N\mu_1(\sigma_i-\sigma_{i-1})\to_p 0$$
and the rest of the proof is essentially the same as that of the first statement; we just replace $N\mu_1(\sigma_2-\sigma_1)+N\mu_1\kappa_2 t$ with $N\mu_1(\sigma_j-\sigma_1)$. 
\end{proof}


In Lemma \ref{nestedballs3}, we establish that $\sigma_{j}^{(2)}-\sigma_{j}>\kappa_{j}\delta$ with high probability.  Then for $k>j$, since $(\sigma_{k}-\sigma_j)/\kappa_{j}\to_p 0$ by Lemma \ref{nestedballs1}, it will follow that $\sigma_{k}-\sigma_j<\kappa_{j}\delta$ with high probability.  It will then follow that $\mathbb{P}(\sigma_{k}<\sigma_j^{(2)})\to 1$, which we show in Lemma \ref{nestedballs5}.

\begin{Lem} \label{nestedballs3}
Suppose (\ref{increasing}) holds, and suppose $\mu_1\gg \alpha/N^{(d+1)/d}$. Let $i \geq 2$.  Define the events $A:=\{\sigma_i-\sigma_{i-1}<\kappa_i t\}$ and $B:=\{\sigma_{i}+\kappa_{i}\delta<\sigma_{i-1}^{(2)}\}$.
Then for any fixed $t>0$ and $\delta>0$, for sufficiently large $N$ we have
$$\mathbb{P}(\sigma_{i}^{(2)}-\sigma_{i}>\kappa_{i}\delta) \geq \exp\Big(-\frac{\gamma_d}{d+1}[(t+\delta)^{d+1}-t^{d+1}]\Big) - \mathbb{P}(A^c) - \mathbb{P}(B^c).$$
\end{Lem}

\begin{proof}
Reasoning as in the proof of Lemma \ref{computeexpectation}, we have
\begin{equation}\label{EY1}
\mathbb{P}(\sigma_{i}^{(2)}-\sigma_{i}>\kappa_{i}\delta)= \mathbb{E} \bigg[ \exp\Big(-\int_{0}^{\kappa_{i}\delta}\mu_{i}Y_{i-1}(s+\sigma_{i})ds\Big) \bigg].
\end{equation}
Because of $\mu_1\gg \alpha/N^{(d+1)/d}$ and (\ref{increasing}), for sufficiently large $N$ we have $\kappa_i(t+\delta)<L/(2\alpha)$. Thus, on the event $A$, we have
$\kappa_i\delta+(\sigma_i-\sigma_{i-1})<L/(2\alpha)$.  Therefore, on the event $A \cap B$, for all $s\in [0,\kappa_i\delta]$, we have for sufficiently large $N$,
\begin{align*}
Y_{i-1}(s+\sigma_i) = \gamma_d\alpha^d(s+\sigma_i-\sigma_{i-1})^{d}\leq \gamma_d\alpha^d(s+\kappa_it)^{d}.
\end{align*}
Thus, for sufficiently large $N$,
\begin{align} \label{expbound}
\exp\Big(-\int_{0}^{\kappa_{i}\delta}\mu_{i}Y_{i-1}(s+\sigma_{i})ds\Big) &\geq \exp\Big(-\int_0^{\kappa_{i}\delta}\mu_{i}\gamma_d\alpha^d(s+\kappa_{i}t)^d ds\Big) \mathbf{1}_{A \cap B} \nonumber \\
&\geq \exp\Big(-\int_0^{\kappa_{i}\delta}\mu_{i}\gamma_d\alpha^d(s+\kappa_{i}t)^d ds\Big) - \mathbf{1}_{A^c} - \mathbf{1}_{B^c}.
\end{align}
It follows from (\ref{EY1}) and (\ref{expbound}) that for sufficiently large $N$,
\begin{align*}
\mathbb{P}(\sigma_{i}^{(2)}-\sigma_{i}>\kappa_{i}\delta) & \geq \exp\Big(-\int_0^{\kappa_{i}\delta}\mu_{i}\gamma_d\alpha^d(s+\kappa_{i}t)^d ds\Big) - \mathbb{P}(A^c) - \mathbb{P}(B^c)
    \\ & = \exp\Big(-\frac{\mu_{i}\gamma_d\alpha^d}{d+1}[(\kappa_{i}\delta+\kappa_{i}t)^{d+1}-(\kappa_{i}t)^{d+1}]\Big) - \mathbb{P}(A^c) - \mathbb{P}(B^c)
    \\ & = \exp\Big(-\frac{\gamma_d}{d+1}[(t+\delta)^{d+1}-t^{d+1}]\Big) - \mathbb{P}(A^c) - \mathbb{P}(B^c),
\end{align*}
as claimed.
\end{proof}



In Lemma \ref{nestedballs5} below, we give sufficient conditions for $\mathbb{P}(\sigma_j<\sigma_i^{(2)})\to 1$ for $j>i\geq 1$, which implies that we obtain nested balls, as in Figure 2 of Section 2.2.

\begin{Lem} \label{nestedballs5}
Assume $\mu_1\gg \alpha/N^{(d+1)/d}$, $\mu_2\gg (N\mu_1)^{d+1}/\alpha^d$, and $\mu_j\ll \mu_{j+1}$ for $j\geq 2$.  Then for every $i\geq 1$ and $j>i$, we have 
\begin{equation}\label{nestedeq}
\mathbb{P}(\sigma_j<\sigma_i^{(2)})\to 1.
\end{equation}
\end{Lem}

\begin{proof}
The result (\ref{nestedeq}) when $i = 1$ was proved in part 2 of Lemma \ref{nestedballs2}.  To establish the result when $i \geq 2$, we will show that for all $i \geq 2$ and all $\epsilon > 0$, there exists $\delta > 0$ such that for sufficiently large $N$, we have
\begin{equation}\label{i2}
\mathbb{P}(\sigma_{i-1}^{(2)} > \sigma_{i} + \kappa_{i} \delta) > 1 - \epsilon
\end{equation}
and
\begin{equation}\label{i12}
\mathbb{P}(\sigma_{i}^{(2)} > \sigma_{i} + \kappa_{i} \delta) > 1 - \epsilon.
\end{equation}
Assume for now that $i \geq 2$ and that (\ref{i2}) and (\ref{i12}) hold.  Let $\eps > 0$, and choose $\delta > 0$ to satisfy (\ref{i12}).  Lemma \ref{nestedballs1} implies that if $j > i$, then $\mathbb{P}(\sigma_j-\sigma_i < \kappa_i \delta) > 1 - \epsilon$ for sufficiently large $N$.  It follows that $$\mathbb{P}(\sigma_j<\sigma_i^{(2)}) \geq \mathbb{P}(\sigma_i^{(2)} > \sigma_i + \kappa_i \delta > \sigma_j) >1-2\epsilon$$ for sufficiently large $N$, which implies (\ref{nestedeq}).

It remains to prove (\ref{i2}) and (\ref{i12}).  We will proceed by induction.  The result (\ref{i2}) when $i = 2$ is part 1 of Lemma \ref{nestedballs2}.  Therefore, it suffices to show that (\ref{i2}) implies (\ref{i12}), and that if (\ref{i12}) holds for some $i \geq 2$, then (\ref{i2}) holds with $i+1$ in place of $i$.

To deduce (\ref{i12}) from (\ref{i2}), we first let $\eps > 0$ and use Lemma \ref{nestedballs0} to choose $t > 0$ large enough that 
$$\mathbb{P}(\sigma_i-\sigma_{i-1}>\kappa_{i}t)\leq\exp\Big(-\frac{\gamma_d}{d+1}t^{d+1}\Big)<\frac{\epsilon}{3}.$$
Then choose $\delta > 0$ small enough that (\ref{i2}) holds with $\epsilon/3$ in place of $\epsilon$ for sufficiently large $N$ and
$$\exp \Big( - \frac{\gamma_d}{d+1}[(t+\delta)^{d+1} - t^{d+1}] \Big) > 1 - \frac{\epsilon}{3}.$$
It now follows from Lemma \ref{nestedballs3} that for sufficiently large $N$,
$$\mathbb{P}(\sigma_{i}^{(2)} > \sigma_{i} + \kappa_{i} \delta) > 1 - \frac{\epsilon}{3} - \frac{\epsilon}{3} - \frac{\epsilon}{3} = 1 - \eps,$$ so (\ref{i12}) holds.

Next, suppose (\ref{i12}) holds for some $i \geq 2$.  Let $\eps > 0$.  By (\ref{i12}), there exists $\delta > 0$ such that for sufficiently large $N$, we have $$\mathbb{P}(\sigma_{i}^{(2)} > \sigma_{i} + \kappa_{i} \delta) > 1 - \frac{\epsilon}{2}.$$
By Lemma \ref{nestedballs1} and the fact that $\mu_i\ll \mu_{i+1}$, we have $(\sigma_{i+1}-\sigma_i)/\kappa_i+\delta(\kappa_{i+1}/\kappa_i)\to_p 0$. Thus, for sufficiently large $N$,
$$\mathbb{P}\Big(\frac{\sigma_{i+1}-\sigma_i}{\kappa_i}+\frac{\kappa_{i+1} \delta}{\kappa_i}<\delta\Big)>1-\frac{\epsilon}{2}$$
and therefore
$$\mathbb{P}(\sigma_i^{(2)}>\sigma_{i+1}+\kappa_{i+1}\delta)\geq \mathbb{P}(\sigma_i^{(2)}-\sigma_i > \kappa_i \delta> \sigma_{i+1}+\kappa_{i+1} \delta-\sigma_i)>1-\epsilon,$$
which is (\ref{i2}) with $i+1$ in place of $i$.
\end{proof}




We now find the limiting distribution of distances between mutations of consecutive types.

\begin{Lem}
\label{pretheorem5}
Suppose $\mu_1\gg\alpha/N^{(d+1)/d}$, $\mu_2\gg (N\mu_1)^{d+1}/\alpha^d$, and $\mu_j\ll \mu_{j+1}$ for $j\geq 2$.  Then for all $s>0$,
\begin{equation}\label{Dj1}
\mathbb{P}\bigg(\frac{D_{j+1}}{\alpha\kappa_{j+1}}\leq s\bigg)\to \int_{0}^{\infty}\gamma_d(t \wedge s)^d\exp\Big(-\frac{\gamma_dt^{d+1}}{d+1}\Big)dt.
\end{equation}
\end{Lem}

\begin{proof}
Define the event
$$A:=\bigcap_{i=1}^{j}\{\sigma_{j+1}<\sigma_i^{(2)}\}.$$
On the event $A$, the first type $j+1$ mutation appears before the second mutation of any type $i \in \{1, \dots, j\}$.  By Lemma \ref{nestedballs5}, we have $\mathbb{P}(A)\to 1$.  As a result, it will be sufficient for us to consider a modified version of our process in which, for $i \in \{1, \dots, j\}$, only the first type $i$ mutation is permitted to occur.  Note that this modified process can be constructed from the same sequence of independent Poisson process $(\Pi_i)_{i=1}^{\infty}$ as the original process.  However, in the modified process, all points of $\Pi_i$ after time $\sigma_i$ are disregarded.  On the event $A$, the first $j+1$ mutations will occur at exactly the same times and locations in the original process as in the modified process.  Therefore, because $\mathbb{P}(A) \rightarrow 1$, it suffices to prove (\ref{Dj1}) for this modified process.  For the rest of the proof, we will work with this modified process, which makes exact calculations possible.

Let $K\in (s,\infty)$ be a constant which does not depend on $N$.  Our assumptions imply that $\mu_{j+1}\gg\mu_1\gg\alpha/N^{(d+1)/d}$. Thus, there is an $N_K$ such that for $N \geq N_K$ we have $\kappa_{j+1}t< L/(2\alpha)$ for all $t\in[0,K]$.  It follows that $Y_{j}(s)=\gamma_d\alpha^d(s-\sigma_j)^d$ for $s\in [\sigma_j,\sigma_j+\kappa_{j+1}K]$.  Therefore, reasoning as in the proof of Lemma \ref{nestedballs0}, we get
\begin{equation}\label{from7}
\mathbb{P}(\sigma_{j+1}-\sigma_j>\kappa_{j+1}t)=\exp\Big(-\frac{\gamma_d}{d+1}t^{d+1}\Big).
\end{equation}
It follows that for $N \geq N_K$, the probability density of $(\sigma_{j+1}-\sigma_j)/\kappa_{j+1}$ restricted to $[0, K]$ is
$$f(t):=\gamma_d t^d\exp\Big(-\frac{\gamma_d}{d+1}t^{d+1}\Big).$$
For $N \geq N_K$ and $t \in [0,K]$, conditional on the event $\{\sigma_{j+1}-\sigma_j=\kappa_{j+1}t\}$, the location of the first type $j+1$ mutation is a uniformly random point on a $d$-dimensional ball of radius $\alpha\kappa_{j+1}t$, which means
$$\mathbb{P}\bigg(\frac{D_{j+1}}{\alpha\kappa_{j+1}}\leq s\Big|\frac{\sigma_{j+1}-\sigma_j}{\kappa_{j+1}}=t \bigg)=\mathbf{1}_{\{s>t\}}+\frac{\gamma_d(\alpha\kappa_{j+1}s)^d}{\gamma_d(\alpha\kappa_{j+1}t)^d}\mathbf{1}_{\{s\leq t\}}=\mathbf{1}_{\{s>t\}}+\frac{s^d}{t^d}\mathbf{1}_{\{s\leq t\}}.$$
It follows that for $N \geq N_K$,
\begin{align*}
\mathbb{P}\bigg(\frac{D_{j+1}}{\alpha\kappa_{j+1}}\leq s \bigg) &= \int_0^K f(t) \bigg( \mathbf{1}_{\{s>t\}}+\frac{s^d}{t^d}\mathbf{1}_{\{s\leq t\}} \bigg) \: dt + \mathbb{P} \bigg( \frac{D_{j+1}}{\alpha\kappa_{j+1}}\leq s, \: \frac{\sigma_{j+1}-\sigma_j}{\kappa_{j+1}} > K \bigg) \\
&= \int_0^K \gamma_d(t \wedge s)^d\exp\Big(-\frac{\gamma_dt^{d+1}}{d+1}\Big) \: dt + \mathbb{P} \bigg( \frac{D_{j+1}}{\alpha\kappa_{j+1}}\leq s, \: \frac{\sigma_{j+1}-\sigma_j}{\kappa_{j+1}} > K \bigg).
\end{align*}
Because (\ref{from7}) implies that $$\lim_{K \rightarrow \infty} \lim_{N \rightarrow \infty} \mathbb{P} \bigg( \frac{D_{j+1}}{\alpha\kappa_{j+1}}\leq s, \: \frac{\sigma_{j+1}-\sigma_j}{\kappa_{j+1}} > K \bigg) = 0,$$
the result (\ref{Dj1}) follows by letting $N \rightarrow \infty$ and then $K \rightarrow \infty$.
\end{proof}

\noindent \textit{Proof of Theorem \ref{theorem5}.} Lemma \ref{pretheorem5} proves the case when $k=j+1$, so assume that $k\geq j+2$.
The triangle inequality implies that 
$$D_{j+1}-(D_{j+2}+\cdots+D_{k})\leq D_{j,k}\leq D_{j+1}+(D_{j+2}+\cdots+D_{k}).$$
Supppose $j+2 \leq i \leq k$.  We know from Lemma \ref{pretheorem5} that $D_i/(\alpha \kappa_i)$ converges in distribution to a nondegenerate random variable as $N \rightarrow \infty$.  Because 
$\kappa_{j+1}/\kappa_i \rightarrow \infty$ by the assumption (\ref{strincreasing0}), it follows that $D_i/(\alpha \kappa_{j+1}) \to_p 0$.
Therefore, $(D_{j+2}+\cdots +D_{k})/(\alpha\kappa_{j+1})\to_p 0$. Thus, Theorem \ref{theorem5} follows from Lemma \ref{pretheorem5} and Slutsky's Theorem. \QED{}

\subsection{Proof of Theorem 6}
\indent Having found a limiting distribution for distances between mutations in the setting of Theorem~2, we now prove a similar result in the setting of Theorem 4, where once the first type $l$ mutation appears, all subsequent mutations appear in nested balls.

We begin with a result that bounds $\sigma_l^{(2)}-\sigma_l$ away from zero with high probability, on the time scale $\beta_l$.

\begin{Lem}
\label{nestedballs9} Assume the same hypotheses as Theorem \ref{theorem4}. Then, for all $\epsilon>0$, there is $r>0$ so that 
$$\liminf_{N\to\infty}\mathbb{P}(\sigma_l^{(2)}-\sigma_l>\beta_l r)>1-\epsilon.$$
\end{Lem}

\begin{proof}
Let $\epsilon>0$. Using Theorem \ref{theorem3}, choose a large $t>0$ so that
\begin{align}
    \label{nb9.0}
    \lim_{N\to\infty}\mathbb{P}(\sigma_l\leq \beta_l t)>1-\frac{\epsilon}{2}.
\end{align}
Now set, as in the proof of Theorem \ref{theorem3}, 
$$g(s):=\frac{\gamma_d^{l-1}(d!)^{l-1}s^{(l-1)(d+1)}}{((l-1)(d+1))!}.$$
It is clear that we can choose a small $r>0$ so that
\begin{align}
    \label{nb9.05}
    \exp\Big(-\int_t^{t+r}g(s)ds\Big)>1-\frac{\epsilon}{2}.
\end{align}
Having chosen $t>0$ and $r>0$, choose $\delta>0$ so that $[t,t+r]\subseteq [\delta,\delta^{-1}]$. Then for any $\lambda>0$, define, as in Lemma \ref{theorem10tool}, the event
$$B:=B_N^{l-1}\Big(\delta,\lambda,g,\frac{1}{\mu_l\beta_l}\Big)=\bigg\{\frac{g(u)(1-\lambda)}{\beta_l\mu_l}\leq Y_{l-1}(\beta_lu)\leq \frac{g(u)(1+\lambda)}{\beta_l\mu_l},\text{ for all }u\in[\delta,\delta^{-1}]\bigg\}.$$
Now we calculate
\begin{align}
    \label{nb9.1}
    \mathbb{P}(\sigma_l^{(2)}-\sigma_l>\beta_l r) & = \mathbb{E}\bigg[\exp\Big(-\int_{\sigma_l}^{\sigma_l+\beta_l r}\mu_lY_{l-1}(s)ds\Big)\bigg] \nonumber
    \\ & \geq \mathbb{E}\bigg[\exp\Big(-\int_{\sigma_l}^{\sigma_l+\beta_l r}\mu_lY_{l-1}(s)ds\Big)\mathbf{1}_{\{\sigma_l\leq \beta_l t\}}\mathbf{1}_{B}\bigg].
\end{align}
Because $Y_{l-1}(s)$ is monotone increasing in $s$, on the event $\{\sigma_l\leq \beta_l t\}\cap B$ we have
$$\int_{\sigma_l}^{\sigma_l+\beta_l r}\mu_lY_{l-1}(s)ds\leq \int_{\beta_l t}^{\beta_l(t+r)}\mu_l Y_{l-1}(s) ds\leq (1+\lambda) \int_{t}^{t+r}g(s)ds.$$
Using the above and (\ref{nb9.1}), we have
$$\mathbb{P}(\sigma_l^{(2)}-\sigma_l>\beta_l r)\geq \exp\bigg(-(1+\lambda)\int_t^{t+r}g(s)ds\bigg)\mathbb{P}(\{\sigma_l\leq \beta_l t\} \cap B)$$
Now take $N\to\infty$. Using that $\mathbb{P}(B)\to 1$ as shown in the proof of Theorem \ref{theorem3}, and using (\ref{nb9.0}), we have
\begin{align*}
    \liminf_{N\to\infty}\mathbb{P}(\sigma_l^{(2)}-\sigma_l>\beta_l r)&\geq\exp\bigg(-(1+\lambda)\int_t^{t+r}g(s)ds\bigg)\cdot \liminf_{N\to\infty}\mathbb{P}(\sigma_l\leq \beta_l t)
    \\ & > \exp\bigg(-(1+\lambda)\int_t^{t+r}g(s)ds\bigg)\Big(1-\frac{\epsilon}{2}\Big).
\end{align*}
Since $\lambda>0$ is arbitrary, (\ref{nb9.05}) implies $\liminf_{N\to\infty}\mathbb{P}(\sigma_l^{(2)}-\sigma_l>\beta_l r)> (1-\epsilon/2)^2 > 1-\epsilon$, finishing the proof.  
\end{proof}

Using Lemma \ref{nestedballs9}, we prove an analog of Lemma \ref{nestedballs2} in the setting of Theorem 4.
\begin{Lem}
\label{nestedballs10}
Assume the same hypotheses as Theorem 4. Then
\begin{enumerate}
    \item For all $t>0$, we have $\mathbb{P}(\sigma_{l+1}+\kappa_{l+1}t<\sigma_{l}^{(2)})\to 1.$
    \item For every $k\geq l+1$, we have 
$\mathbb{P}(\sigma_{k}<\sigma_l^{(2)})\to 1.$
\end{enumerate}
\end{Lem}

\begin{proof}
Let $\epsilon>0$. Lemma \ref{nestedballs9} implies that there is $r>0$ and such that for sufficiently large $N$,
\begin{equation}\label{141}
\mathbb{P}(\sigma_l^{(2)}-\sigma_l>\beta_lr)>1-\epsilon.
\end{equation}
Now note that $(\sigma_{l+1}-\sigma_l)/\beta_l\to_p 0$ by (\ref{sigmaconvprob}) in the proof of Theorem \ref{theorem4}. Also, our assumption that $\mu_{l+1}\gg 1/(\alpha^d\beta_l^{d+1})$ is equivalent to $\kappa_{l+1}\ll \beta_l$. It follows that for sufficiently large $N$,
\begin{equation}\label{142}
\mathbb{P}((\sigma_{l+1}-\sigma_l+\kappa_{l+1}t)/\beta_l<r)>1-\epsilon.
\end{equation}
The estimates (\ref{141}) and (\ref{142}) imply that $\mathbb{P}(\sigma_{l+1}+\kappa_{l+1}t<\sigma_l^{(2)})>1-2\epsilon$ for sufficiently large $N$.  This proves the first statement. The second statement is proven similarly, using instead that $(\sigma_k-\sigma_l)/\beta_l\to_p 0$ by (\ref{sigmaconvprob}). 
\end{proof}

At this point, we have proven that for $k>l$, the first type $k$ mutation occurs before $\sigma_l^{(2)}$ with probability tending to $1$ as $N\to \infty$. This implies that in the setting of Theorem 4, we can disregard the type $1,...,l-1$ mutations and regard the first type $l$ mutation as the first type~$1$ mutation, and then prove Theorem \ref{theorem6} by following the same argument used to prove Theorem \ref{theorem5}.
\\~\\
\textit{Proof of Theorem 6.}
Relabel the type $l,l+1,l+2,...$ mutations as type $1,2,3,...$ mutations, and repeat the arguments in Lemmas \ref{nestedballs0}-\ref{pretheorem5} and in the proof of Theorem \ref{theorem5}. The only difference is that we have to apply Lemma \ref{nestedballs10} instead of Lemma \ref{nestedballs2}. Note that type $l$ mutations do not appear at the same rate as type $1$ mutations, so we needed a different technique to establish $\mathbb{P}(\sigma_j<\sigma_l^{(2)})\to 1$ for $j>l$. \QED{}  
\\~\\
\textbf{Acknowledgments.} The authors thank Jasmine Foo for suggesting the problem of looking at the case of increasing mutation rates, and bringing to their attention the references \cite{ll00, pfl10}.  JS was supported in part by NSF Grant DMS-1707953.

\end{document}